\documentclass{article}

\usepackage[margin=1.2in]{geometry} 

\usepackage{amsmath, amsthm, amssymb, amsfonts} 
\usepackage{mathtools} 
\usepackage{yhmath}    

\usepackage{yfonts} 

\usepackage{graphicx}
\usepackage{epsfig, epstopdf} 
\usepackage{tikz} 
\usetikzlibrary{fadings, patterns, shadows.blur, shapes, tikzmark, calc, decorations.pathreplacing} 

\usepackage{array, multirow, tabularx, booktabs} 
\usepackage{enumitem}
\usepackage{hyperref} 
\usepackage{caption}

\hypersetup{
    colorlinks=true,
    linkcolor=blue!80!black,,
    citecolor=blue!50!black,,
    urlcolor=blue
}
\usepackage{titling} 
\usepackage{url} 
\usepackage{dsfont} 
\usepackage{physics} 
\usepackage{cancel} 
\usepackage{color} 
\usepackage{siunitx} 
\usepackage{gensymb} 
\usepackage{extarrows} 
\usepackage{wasysym} 
\usepackage{multicol} 
\usepackage{ytableau} 
\usepackage{float} 
\usepackage[utf8]{inputenc} 
\usepackage{biblatex} 

\theoremstyle{remark}  

\newtheorem{theorem}{Theorem}[section]

\newtheorem{lemma}[theorem]{Lemma}

\newtheorem{definition}[theorem]{Definition}

\newtheorem{example}[theorem]{Example}

\newtheorem{Con}[theorem]{Conjecture}
 
\newcommand{\eig}{\text{eig}} 
\newcommand{\R}{\mathbb{R}}  

\title{Shuffling via sums of Jucys--Murphy Elements}
\author{
Samira Arfaee \thanks{Stony Brook University, NY 11794. E-mail: \texttt{seyedehfatemeh.arfaeezarandi@stonybrook.edu}. Supported by the NSF grant DMS-2450510.} 
\and 
Evita Nestoridi \thanks{Stony Brook University, NY 11794. E-mail: \texttt{evrydiki.nestoridi@stonybrook.edu }. Supported by the Simons Foundation Travel Support for Mathematicians -- MPSTSM00007955 and the NSF grant DMS-2450510.}
}

\date{}

\begin{document}

\maketitle  

\begin{abstract}

We consider a family of card shuffles of $n$ cards in which the allowed moves involve transpositions corresponding to the Jucys--Murphy elements of the symmetric group $\{S_m\}_{m \leq n}$. We determine the eigenvalues of the corresponding   $n! \times n!$ transition matrices of these shuffles and study the mixing times for a special case, the $k$--star transpositions shuffle, a natural interpolation between the random transpositions
\cite{Diaconis1981} and the star transpositions \cite{PFlour}. We prove that the $k$--star transpositions shuffle exhibits total variation cutoff at $\frac{2n-(k+1)}{2(n-1)}n\log n$ with a window of $\frac{2n-(k+1)}{2(n-1)}n$. Furthermore, in the regimes $k/n \rightarrow 0$ or $1$, this shuffle has the same limit profile as random transpositions, which has been fully determined by Teyssier \cite{Teyssier2019}.

\end{abstract}


\section{Introduction}

Shuffling a deck of $n$ cards via transpositions has been widely studied \cite{10.1214/20-AAP1632, BD, AC, PFlour, Diaconis1981, ghosh2019total, JJ, Lacoin, NN, Teyssier2019}. In their seminal work, Diaconis and Shahshahani \cite{Diaconis1981} proved that it takes $\frac{1}{2} n \log n$ steps to shuffle a deck of $n$ cards by random transpositions. Diaconis \cite{PFlour} also proved that shuffling via the star transpositions takes $n \log n$ steps. Both works rely on diagonalizing the corresponding transition matrices using the representation theory of the symmetric group \cite{Diaconis1981, FOW}. In this paper, we diagonalize and study the mixing properties of different families of transposition shuffles that interpolate between the star transpositions and the random transpositions shuffle.  Although our diagonalization is motivated by the general framework of \cite{ ,axelrod2024spectrum,10.1214/20-AAP1632, Dieker2018, OMR}, the special role of the Jucys--Murphy elements in our family
of shuffles permits a more elementary and direct proof, relying on the
properties established in \cite{Mathas, MURPHY1992492}.

Let $2 \leq j \leq n$, and let $S_j$ be the symmetric group on
$[j] := \{1,\ldots,j\}$. Let
$
T_j := \{(i,j) \mid 1 \leq i < j\}.
$
Then, the $j$-th Jucys--Murphy element is
$
J_j := \sum_{\tau \in T_j} \tau
$.


Consider $E \subset [n]$ such that $n \in E$. The corresponding set of transpositions is defined as
$ T_E= \cup_{i \in E} T_i $
and the transition matrix is
\begin{equation}\label{transition}
P_E(x,  \sigma x)= \begin{cases}
\frac{1}{n}, & \sigma=\textup{id},\\
\frac{n-1}{n} \frac{1}{|T_E|}, & \textup{ if } \sigma \in T_E,\\
0, & \textup{otherwise,}
\end{cases}
\end{equation}
where $x, \sigma \in S_n$. We can also write this as an element of the group algebra $\mathbb{C}[S_n]$, as a linear sum of Jucys--Murphy elements \begin{equation}\label{TranJucy}
P_{E}=\frac{1}{n}I+\frac{n-1}{n|T_E|}\sum_{j \in E}J_j.
\end{equation}
The eigenvalues of $P_E$ are indexed by the set of standard Young tableaux of $n$. Let $\lambda= (\lambda_1, \lambda_2, \ldots)$ be a partition of $n$ and write $\lambda \vdash n$. Let $SYT(\lambda)$ be the set of standard Young tableaux of shape $\lambda$ and let $d_{\lambda}= |SYT(\lambda)|$.

\begin{theorem}\label{T-EVRRT}
Let $S \in SYT(\lambda)$ and $S(i,j)$ denote the number in box $(i,j)$ of $S$. The eigenvalue of $P_E$ corresponding to $S$ is
\begin{equation}\label{eig
}
\eig(S)= \frac{1}{n}+\frac{(n-1)}{n|T_E|}\sum_{S(i,j)\in E}  (j-i).
\end{equation}

\end{theorem}

Let $k \in [n]$ and let $E=\{n-k+1, \ldots, n\}$. This is the case of the $k$--star transpositions card shuffle, where the eigenvalues take a simpler expression in terms of the diagonal index of a partition. More precisely, we think of a partition $\lambda$ of $n$ as a diagram with $\lambda_i$ boxes on the $i$--th row. Under this convention, $(i,j)$  corresponds to the box in row $i$ and column $j$ of $\lambda$. The diagonal index of $\lambda$ is then defined as
$
\textup{Diag}(\lambda)  = \sum_{(i,j) \in \lambda} (j - i).
$

\begin{theorem}\label{T-EVRST}
Let $\lambda \vdash n$ and let $\mu \vdash n-k$, such that $\mu \subseteq \lambda$. The eigenvalue of the $k$--star shuffle corresponding to the pair $(\lambda, \mu)$ is 
\begin{equation} \label{formula}  
\eig(\lambda,\mu)= \frac{1}{n} + \frac{2(n-1)}{nk(2n-(k+1))} \bigg(\textup{Diag}(\lambda)-\textup{Diag}(\mu)\bigg),
\end{equation}
 with multiplicity $d_{\lambda}d_{\mu}d_{\lambda
/ \mu }$.
\end{theorem}

Theorems \ref{T-EVRRT} and \ref{T-EVRST} are proved by combining ideas from \cite{Mathas, MURPHY1992492}. One could also prove them using the lifting eigenvectors technique for analyzing shuffles, which was introduced by Dieker and Saliola in \cite{Dieker2018} to study the random-to-random shuffle (whose mixing behavior was studied in \cite{bernstein2019cutoff} and \cite{Subag}). This technique has now been applied in different setups \cite{axelrod2024spectrum, 10.1214/20-AAP1632, BCGS, EZLT,  GL, NP}. The first ones to consider applying this technique for a set of transpositions were Bate, Connor, and Matheau--Raven in \cite{10.1214/20-AAP1632}, when studying the cutoff for the one--sided transposition shuffle. The properties of the Jucys--Murphy elements discussed in \cite{Mathas, MURPHY1992492} allow us to implement simpler arguments. 

We study the mixing time and the limit profile of the $k$--star transpositions through their spectrum. Let $P_k$ denote the transition matrix of the $k$--star transpositions and let
$$\Vert P_{k}^t(x, \cdot)-U \Vert _{\textup{T.V.}}:= \frac{1}{2} \sum_{y \in S_n} \bigg \vert P_{k}^t(x, y) - \frac{1}{n!} \bigg \vert $$
be the total variation distance starting at $x \in S_n$. The total variation distance is defined as $$d(t)= \max_{x \in S_n} \Vert  P_{k}^t(x, \cdot)-U\Vert _{\textup{T.V.}}. $$
\begin{theorem}\label{T-Cutoff}
Let $ k \in [n]$ and set $t_{n,k}(c)=\frac{2n-(k+1)}{2(n-1)}n(\log n+c)$. For the $k$--star transpositions, we have that 
$$\lim_{c \rightarrow \infty} \lim_{ n \rightarrow \infty} d(t_{n,k}(c)) =0 \mbox{ and }\lim_{c \rightarrow - \infty} \lim_{ n \rightarrow \infty} d(t_{n,k}(c)) =1.$$
\end{theorem}

In other words, the $k$--star transpositions shuffle exhibits a total variation cutoff at 
$\frac{2n-(k+1)}{2(n-1)}n\log n$ with window $\frac{2n-(k+1)}{2(n-1)}n$.

We also study the limit profile of the $k$--star transpositions, defined as
$$\Phi_{k}(c):= \lim_{ n \rightarrow \infty} d(t_{n,k}(c)), $$
when this limit exists.
Teyssier \cite{Teyssier2019} derived an explicit formulation for the limit profile of random transpositions, which corresponds to $k=n$. The star transpositions shuffle ($k=1$) has the same limit profile as the random transpositions shuffle, as shown in \cite{nestoridi2024comparing}. In general, we can consider $k$ as a function of $n$. Let $k=f(n)$; then we can study $\Phi_{k=f(n)}(c)$.  
We use the comparison method introduced in \cite{nestoridi2024comparing} to extend this result.

\begin{theorem}\label{T-LP}
Let $k$ be such that $\lim_{n \rightarrow \infty} \frac{k}{n}=0$ or $1$. For the $k$--star transpositions card shuffle at time $t_{n,k}(c)$, we have
$$\Phi_{k}(c)=d_{\textup{T.V.}}(\textup{Poiss}(1+e^{-c}),  \textup{Poiss}(1)),$$
for all $c \in \R$.

\end{theorem}
The restriction on $k$ comes from the variation in the multiplicities of the eigenvalues. The comparison technique fails to give the desired result for any $k$, but we still conjecture that the above limit behavior holds for general $k$.
\begin{Con}
For the $k$--star transpositions card shuffle, if $\lim_{n \rightarrow \infty} \frac{k}{n}=\alpha<1$, then we have
$$\Phi_{k}(c)=d_{\textup{T.V.}}(\textup{Poiss}(1+(1-\alpha) e^{-c}),  \textup{Poiss}(1)),$$
for all $c \in \R$.
\end{Con}

The paper is organized as follows: Section \ref{prel} gives all the important definitions and tools needed from representation theory in order to prove Theorems \ref{T-EVRRT}, \ref{T-EVRST}, \ref{T-Cutoff}, and \ref{T-LP}. Section \ref{eigen} presents the proof of Theorems \ref{T-EVRRT} and \ref{T-EVRST}. Section \ref{boundeigen} provides bounds for the eigenvalues of $P_k$, which are later used in Section \ref{upper}
to prove the upper bound of Theorem \ref{T-Cutoff}. Section \ref{lower} presents the proof of the lower bound of Theorem \ref{T-Cutoff}. Section \ref{lp} discusses the limit profile of the $k$--star and proves Theorem \ref{T-LP}.

\section{Preliminaries}\label{prel}
In this section, we give all necessary definitions borrowed from the representation theory of the symmetric group.

\subsection{Representation theory of the symmetric group}
The irreducible representations of $S_n$ are indexed by partitions of $n$, defined as follows.
\begin{definition}
     A partition $\lambda$ of a positive integer $n$ (denoted by $\lambda \vdash n$) with $m$ parts is a vector  $\lambda = (\lambda_1, \lambda_2, \dots, \lambda_m)$ with integer coordinates, where $\lambda_1 \geq \lambda_2 \geq \dots \geq \lambda_m > 0$ and $\sum_{i=1}^{m} \lambda_i = n$. 
\end{definition}
Every partition can be depicted as a Young diagram, which has $\lambda_i$ boxes in row $i$.
\begin{definition}

Let $\lambda = (\lambda_1, \lambda_2, \dots, \lambda_m)$ and $\mu = (\mu_1, \mu_2, \dots, \mu_{\ell})$ be two partitions of the same integer $n$. We say that $\lambda$ dominates $\mu$ (denoted $\lambda \trianglerighteq \mu$) if:

$$
\sum_{i=1}^l \lambda_i \geq \sum_{i=1}^l \mu_i \quad \text{for all } l = 1, 2, \dots, \min \{\ell, m\}.$$

\end{definition}

\begin{example}
Let $\lambda=(4, 4)$ and $\mu= (4, 3, 1) $, then $ \lambda \trianglerighteq \mu $. The corresponding Young diagrams are:
\begin{center}
\begin{minipage}{6cm}
\centering
$\lambda=(4, 4)$
$$
\begin{ytableau}
\ & \ & \ & \\\
\ & \ & \ & \ \\
\end{ytableau}
$$
\end{minipage}
\begin{minipage}{6cm}
\centering
$\mu= (4, 3, 1) $
$$
\begin{ytableau}
\ & \ & \ & \ \\
\ & \ & \ \\
\  \\
\end{ytableau}
$$
\end{minipage}
\end{center}
\end{example}

\

\begin{definition}

The content of a box $(i,j)$ in a Young diagram is defined as:
$$
d_{(i,j)} = j - i,
$$
where:
\begin{enumerate}
    \item $ i $ is the row number of the box (starting from $ 1 $),
    \item $ j $ is the column number of the box (starting from $ 1 $).
\end{enumerate}

The diagonal index of a Young tableau is the sum of the contents of all boxes in the tableau, namely
$$
\textup{Diag}(T)= \sum_{(i,j) \in T} d_{(i,j)} = \sum_{(i,j) \in T} (j - i),
$$
where $ T $ represents the set of all the boxes in the Young tableau.
\end{definition}

For example, consider a Young tableau of shape $ \lambda = (4, 3, 1) $. The contents of the boxes of $\lambda$ are given in the following picture:
$$
\begin{ytableau}
0 & 1 & 2 & 3 \\
-1 & 0 & 1 \\
-2
\end{ytableau}
$$

The analysis of the eigenvalues is smoother if we shift these contents in the following manner, so that the shifted values are all positive.
\begin{definition}
Let $\lambda = (\lambda_1, \lambda_2, ..., \lambda_m) \vdash n$. The $\lambda_1$-shifted content of a box in the $i$-th row and $j$-th column of the tableau is defined as:

$$
\text{$\lambda_1$-shifted content } = \lambda_1 - (j - i) = \lambda_1 - d_{(i,j)}.
$$

\end{definition}

For example, consider a Young tableau of shape $ \lambda = (4, 3, 1) $. The $\lambda_1$-shifted contents of the boxes of the tableau are:

$$
\begin{ytableau}
4 & 3 & 2 & 1 \\
5 & 4 & 3 \\
6
\end{ytableau}
$$

To discuss the formulas and the multiplicities of the eigenvalues, we need the notion of a standard Young tableau of shape $\lambda$ (see the following definition).

\begin{definition}
 A standard Young tableau is a filling of a Young diagram with the numbers $1, 2, \ldots, n$ such that entries increase along each row and each column. Denote by $SYT(\lambda)$ the set of standard Young tableaux of shape $\lambda $ and let $d_{\lambda}= |SYT(\lambda)|$.  
\end{definition}

\begin{example}
The partition $\lambda=(4,3,1)$ corresponds to the Young diagram and a standard Young tableau
\begin{center}
\begin{minipage}{6cm}
\centering
Young Tableau
$$
\lambda=
\ytableausetup{centertableaux}
\begin{ytableau}
\text{ } & \text{ } & \text{ }  & \text{ } \\
\text{ } & \text{ } & \text{ }\\
\text{ }\\
\end{ytableau}
$$
\end{minipage}
\begin{minipage}{6cm}
\centering
Standard Young Tableau
$$
\ytableausetup{centertableaux}
S=\begin{ytableau}
1 & 2 & 4 & 6 \\
3 & 5 & 8\\ 
7\\
\end{ytableau}  
$$
\end{minipage}
\end{center}
\end{example}


\begin{definition}[Jucys--Murphy elements of $S_n$]\label{def:Jucys-Murphy}
For $2 \leq m \leq n,$  the $m$--th Jucys--Murphy element is
    $$ J_m=\sum_{i=1}^{m-1} (i, m), $$
where we view the transposition $(i,m)$ as a map from $S_n$ to $S_n$ that sends $\sigma \in S_n$ to $\sigma \cdot (i,m)$. 
\end{definition}

Note that $J_m$ can be viewed as an $n! \times n!$ matrix.

The eigenvalue corresponding to $S \in SYT(\lambda)$ is given as an expression of the following statistic.
\begin{definition} Let $\lambda=(\lambda_1, \dots, \lambda_m)$ be a partition of $n$ and $S \in SYT(\lambda)$. We define the $k$--diagonal index of $S$ as follows:
 $$D_{S}^{k}=\sum_{n-k<S(i,j) \leq n} (j-i),$$
 which is the sum of the contents of the boxes containing the numbers $n-k+1, \ldots, n$ in $S$.
\end{definition}

Similarly, we consider the shifted $k$--diagonal index, which will make the analysis better.
\begin{definition}\label{shifted}
Let $\lambda=(\lambda_1, \dots, \lambda_m)$ be a partition of $n$ and $S \in SYT(\lambda)$, then we define 
  $$A_{S}^{k}=k\lambda_1-D^{k}_{S}=\sum_{n-k< S(i,j) \leq n} (\lambda_1-(j-i)).$$
\end{definition}


To bound the eigenvalues of the transition matrix, we look at two special cases of standard Young tableaux.
\begin{definition}[Row-Insertion Tableau]\label{rinsertion}
Let $\lambda  \vdash n$. The row-insertion tableau, denoted $T_{\lambda^{\to}}$, of shape $\lambda$ is formed by inserting the numbers $1, 2, \dots, n$ row by row, from left to right, starting from the top row. 
\end{definition}
Similarly, we define the column insertion tableau.
\begin{definition}[Column-Insertion Tableau]\label{cinsertion}
Let $\lambda   \vdash n$. The column-insertion tableau, denoted $T_{\lambda^{\downarrow}}$, of shape $\lambda$ is formed by inserting the numbers $1, 2, \dots, n$ column by column, from top to bottom, starting from the leftmost column.
\end{definition}

\begin{example}
For $ \lambda = (4,3,1) $ we have:

\begin{center}
\begin{minipage}{6cm}
\centering
Row-Insertion Tableau ($ T_{\lambda^{\to}} $):
$$
\begin{ytableau}
1 & 2 & 3 & 4 \\
5 & 6 & 7\\
8\\
\end{ytableau}
$$
\end{minipage}
\begin{minipage}{6cm}
\centering
Column-Insertion Tableau ($ T_{\lambda^{\downarrow}} $):
$$
\begin{ytableau}
1 & 4 & 6 & 8  \\
2 & 5 & 7\\
3
\end{ytableau}
$$
\end{minipage}
\end{center}
\end{example}

To study the multiplicity of eigenvalues of the $k$--star transpositions, we need to introduce the notion of skew Young tableaux.
\begin{definition} Let $\lambda \vdash n$ and $\mu \vdash n-k$ be such that $\mu \subseteq \lambda$ when viewed as diagrams.
A skew shape (or skew Young tableau) $\lambda/\mu$ is the collection of boxes obtained by removing the boxes of $\mu$ from the Young diagram of $\lambda$.
\end{definition}
We will also consider fillings of skew diagrams as explained below.
\begin{definition}
Let $\lambda \vdash n$ and $\mu \vdash n-k$. A standard skew tableau is a filling of $\lambda/\mu$ with numbers from $1, \dots, k$ so that the labels of the rows and columns are in increasing order.
Let $d_{\lambda/ \mu}$ be the number of standard skew Young tableaux of shape $\lambda /\mu$.  
\end{definition}

\begin{example}
The partitions $\lambda=(4,3,2)$ and $\mu=(2,1)$ correspond to a skew shape and a skew tableau of shape $\lambda / \mu$
\begin{center}
\begin{minipage}{6cm}
\centering
Skew Shape
$$
\lambda/\mu=
\ytableausetup{centertableaux}
\begin{ytableau}
\none & \none & \text{ } & \text{ } \\
\none & \text{ } & \text{ }\\
\text{ } & \text{ }\\
\end{ytableau}
$$
\end{minipage}
\begin{minipage}{6cm}
\centering
Skew Tableau
$$
\ytableausetup{centertableaux}
S_{\lambda / \mu} =\begin{ytableau}
\none & \none & 1 & 3 \\
\none & 2 & 5\\
4 & 6\\
\end{ytableau}
$$
\end{minipage}
\end{center}
\end{example}

\subsection{The $\ell_2$ bound}   
In this section, we explain the connection between the eigenvalues of $P_E$ and the total variation distance mixing. We adapt the following lemma, borrowed from \cite{LPW-MCMT}, to the case of $S_n$.

\begin{lemma}[Lemma 12.6 \cite{LPW-MCMT}]
    Let $P$ be a reversible, irreducible and aperiodic transition matrix of a random walk on $S_n$. Then the eigenvalues of $P$ satisfy
$$-1< \xi_{n!-1}\leq \xi_{n!-2} \leq \cdot\cdot\cdot \leq \xi_1 < \xi_0=1 $$ 
and 

\begin{equation}\label{l2}
     2\Vert P_{x}^{t}-U\Vert_{T.V.} \leq \left(\sum_{i=1}^{n!-1} \xi_{i}^{2t}\right)^{1/2},
     \end{equation}
    where the sum is over non-one eigenvalues of transition matrix $P$. 

\end{lemma}


We note that, as shown in Theorem 6, Chapter 3E of \cite{Diaconis1988}, the eigenvalues of the transition matrix $P_E$ are indexed by the irreducible representations of $S_n$.

\begin{lemma}[Proposition 1.10.1 \cite{E-S}]
For the symmetric group $S_n$ on $n$ letters, 
$\sum_{\lambda \vdash n} d_{\lambda}^2 = n! $
and 
$
d_{\lambda} \leq \sqrt{n!},
$ for every $\lambda \vdash n$.
\end{lemma}
 We will also use the following, more refined bound.
\begin{lemma}[Corollary 2, \cite{Diaconis1981}]\label{L-DUP}
Let $\lambda = (\lambda_1, \lambda_2, ..., \lambda_m)$ be a partition of $n$.
\begin{equation*}
    d_{\lambda} \leq \binom{n}{\lambda_1}\sqrt{(n-\lambda_1)!}
\end{equation*}
\end{lemma}
To bound the multiplicities of the eigenvalues, we will also need the following asymptotic formula of the number of partitions $p(n)$ of $n$ which was proven in \cite{Hardy2000}.  
\begin{lemma} [Hardy-Ramanujan]\label{T-HR}    \phantom{h}
    $$
        p(n) \sim \frac{1}{4n\sqrt{3}} \exp \left\{ \pi \sqrt{\frac{2n}{3}} \right\}.
    $$
\end{lemma}










\subsection{Limit profiles}
The main tool for studying the limit profile of the $k$--star transpositions is the following adaptation of Lemma 1.4 from \cite{nestoridi2024comparing}.

\begin{lemma}[Lemma 1.4, \cite{nestoridi2024comparing}]\label{comp}
Let $P$ and $Q$ be symmetric transition matrices of two Markov chains on $S_n$ that share the same eigenbasis.  Let $\beta_i$ and $ q_i$ be the eigenvalues of $P$ and $Q$ that respectively correspond to the same $i$-th eigenvector, with $\beta_1 = q_1 = 1$. Assume $P$ exhibits cutoff at $t_n$ with window $w_n$ and has limit profile $\Phi$ and similarly that $Q$ exhibits cutoff at $\bar{t}_n$ with window $\bar{w}_n$ and has limit profile $\bar{\Phi}$. Let $t= t_n+cw_n$ and $\bar{t}= \bar{t}_n+c\bar{w_n}$. We have
$$ |\Phi(c) - \bar{\Phi}{}(c)| \leq \frac{1}{2} \lim_{n\to\infty} \left( \sum_{i=2}^{|X|} (\beta_i^t - q_i^{\bar{t}})^2 \right)^{1/2}. $$

\end{lemma}

Lemma \ref{comp} is instrumental in proving Theorem \ref{T-LP}. In particular, we compare the limit profile of the $k$--star transpositions with the limit profile of random transpositions. For this reason, we borrow the following result.
\begin{theorem}[Theorem 1.1 \cite{Teyssier2019}] For the random transpositions card shuffle at time $t = \frac{n}{2}(\log n + c)$, we have that
\begin{equation*}
\Phi(c) = d_{\textup{T.V.}}(\text{Poiss}(1 + e^{-c}), \text{Poiss}(1)),
\end{equation*}
for every $c \in \mathbb{R}$.
\end{theorem}

\section{The eigenvalues} \label{eigen}

In this section, we discuss the proofs of Theorems \ref{T-EVRRT} and \ref{T-EVRST}. We drop the assumption $n \in E$, which was needed only for irreducibility. 
Let $E$ be a non--decreasing sequence of subsets of $[n]$, as $n$ varies. Recall the definition of the transition matrix $P_{E}$ given in \eqref{TranJucy}.
\begin{equation*}
P_{E}=\frac{1}{n}I+\frac{n-1}{n|T_E|}\sum_{j \in E}J_j.
\end{equation*}


We begin by summarizing some properties of Jucys--Murphy elements, proved in proposition 3.26(iii) \cite{Mathas}, page 506 of \cite{MURPHY1992492} and proposition 2.11(ii) \cite{BCGS} (for $q=1)$.
\begin{lemma}\label{Jucy}
Jucys--Murphy elements have the following properties:
\begin{enumerate}

\item They pairwise commute.
\item Let $S \in SYT(\lambda)$ and m be in the box $(i_m,j_m)$ of $S$. The eigenvalue of $J_m$ corresponding to $S$ is

\begin{equation}
\eig_{J_m}(S)=j_m-i_m.
\end{equation}
\end{enumerate}
\end{lemma}

We now compute the eigenvalues of $P_E$. 
\begin{proof}[Proof of \text{ Theorem} \ref{T-EVRRT}]
Using Lemma \ref{Jucy}, Jucys--Murphy elements commute pairwise and also if $S(i,j)=m$, then we get $\eig_{J_m}(S)=j-i$. Therefore, we have 
\begin{equation*}
\eig(S)= \frac{1}{n}+\frac{(n-1)}{n|T_{E}|}\sum_{ S(i,j) \in E}  (j-i) .
\end{equation*}


\end{proof}


These eigenvalues take a simpler form when $E=\{n-k+1, \ldots, n\}$.
\begin{proof}[Proof of \text{ Theorem} \ref{T-EVRST}]

We now adapt the proof of Theorem \ref{T-EVRRT}, by setting $E=\{n-k+1, \dots, n\}$. Therefore, we have   

\begin{equation*}
\eig(S)= \frac{1}{n}+\frac{2(n-1)}{nk(2n-(k+1))}\sum_{n-k+1 \leq S(i,j) \leq n}  (j-i) .
\end{equation*}

Let $\mu$ be the partition of $n-k$ that is obtained from $\lambda$ and $S$ by removing the boxes of $\lambda$ that contain the numbers $n-k+1, \dots, n$. Removing these boxes also yields a standard Young tableau of shape $\mu$. In this way, we can index the eigenvalues of the $k$--star transpositions by pairs $(\lambda, \mu)$ 
where
$\lambda \vdash n $, $ \mu \vdash n-k$
  and $\mu \subseteq \lambda$. The corresponding formula is

\begin{equation*}   
\eig(\lambda,\mu)= \frac{1}{n} + \frac{2(n-1)}{nk(2n-(k+1))} \bigg(\textup{Diag}(\lambda)-\textup{Diag}(\mu)\bigg),
\end{equation*}
and the multiplicity of this eigenvalue
is $d_{\lambda}d_{\mu}d_{\lambda / \mu }$. This is because we count the number of ways the numbers $1$ through $n-k$ appear in $\mu$ which gives us $d_{\mu}$ and also the number of ways the labels $n-k+1, \dots, n$ can be arranged in the skew shape $\lambda / \mu$, which gives $d_{\lambda / \mu} $.
\end{proof}



\section{Bounding the eigenvalues} \label{boundeigen}  
In this section, we focus on the case $E=\{ n-k+1,\dots, n\}$. We present a few bounds on the eigenvalues that will help us with the analysis of \eqref{l2}. Recall the formulas of the eigenvalues of the $k$--star transpositions given in Theorem \ref{T-EVRST} and the definitions of $T_{\lambda^{\to}}$ and $T_{\lambda^{\downarrow}}$ from Definitions \ref{rinsertion} and \ref{cinsertion}.
\begin{lemma}\phantom{}
\label{lem:eig-bounds} Let $\lambda \vdash n$ and let $S \in SYT(\lambda)$. We have
        $$\mathrm{eig}\bigl(T_{\lambda^{\to}}\bigr)
        \leq 
        \mathrm{eig}(S)
        \leq
        \mathrm{eig}\bigl(T_{\lambda^{\downarrow}}\bigr),
        $$
        where 
        $\displaystyle T^{k}_{\lambda^{\downarrow}} \in SYT(\lambda)$ is the \emph{column-insertion tableau} and     $\displaystyle T^{k}_{\lambda^{\to}} \in SYT(\lambda)$ is the \emph{row-insertion tableau} as in Definitions \ref{rinsertion} and \ref{cinsertion}.

\end{lemma}

\begin{proof}\phantom{}
We will only prove that $\mathrm{eig}(S)
        \leq
        \mathrm{eig}\bigl(T^{k}_{\lambda^{\downarrow}}\bigr)$, since the other inequality is similar. We proceed by double induction on $(n, k)$. 
When $k=1$, we are in the star transpositions case. In this case, column-insertion  of the value $n$ places it in the highest possible row, which indeed maximizes the eigenvalue. So the inequality is true for $(n,1)$.
Assume now that the statement holds for any pair $(m, l)$ where $m < n$ and $ 1 \leq l \leq m $. Consider the pair $(n, k)$ with $k>1$ and consider $T \in SYT(\lambda)$ that maximizes the eigenvalue.

Case 1: $n$ is in the last column of $T$.
Remove the box containing $n$. By the inductive hypothesis for $(n-1, k-1)$, column-insertion maximizes the eigenvalue in the resulting tableau. This gives that $T= T_{\lambda^{\downarrow}}$.

Case 2: $n$ is not in the last column of $T$.
Remove the box containing $n$. The resulting standard Young tableau must maximize the eigenvalue for the corresponding shape for $(n-1, k-1)$, since adding the content of the removed box does not change the order of the eigenvalues. The inductive hypothesis for $(n-1, k-1)$ gives that $n-1$ must be in the last column of $\lambda$. The assumption $k>1$ guarantees that neither $n$ nor $n-1$ are inside $\mu$ from the eigenvalue formula given by equation \eqref{formula}. Since $n$ and $n-1$ i not in the same column or row, we can switch the positions of $n$ and $n-1$ without changing the eigenvalue (this move does not affect $Diag(\lambda)$ or $Diag(\mu)$). Now, $n$ is in the last column, reducing this to Case 1.

By the principle of double induction, the statement holds for all pairs $(n, k)$.
\end{proof}

Lemma \ref{lem:eig-bounds} says that to bound all eigenvalues $\mathrm{eig}(S)$, where $S \in SYT(\lambda)$, we should bound $\mathrm{eig}\bigl(T_{\lambda^{\downarrow}}\bigr)$, which is equal to
$$ \frac{1}{n} + \frac{2(n-1)}{nk(2n-(k+1))} D_{\lambda^\downarrow}^{k}. $$

To maximize $D_{\lambda^\downarrow}^{k}$, we should minimize $A_{\lambda^\downarrow}^{k}$, the $\lambda_1$-shifted diagonal index introduced in Definition \ref{shifted}. Focusing on the standard Young tableau $T^{k}_{\lambda^{\downarrow}}$ and the $k$ boxes that are removed to get $\mu$ (these are the boxes of $T^{k}_{\lambda^{\downarrow}}$ that contain the values $n-k+1$ through $n$), we see that the $\lambda_1$-shifted index of each box is given by the following picture.

\begin{center}
    \includegraphics[width=\linewidth]{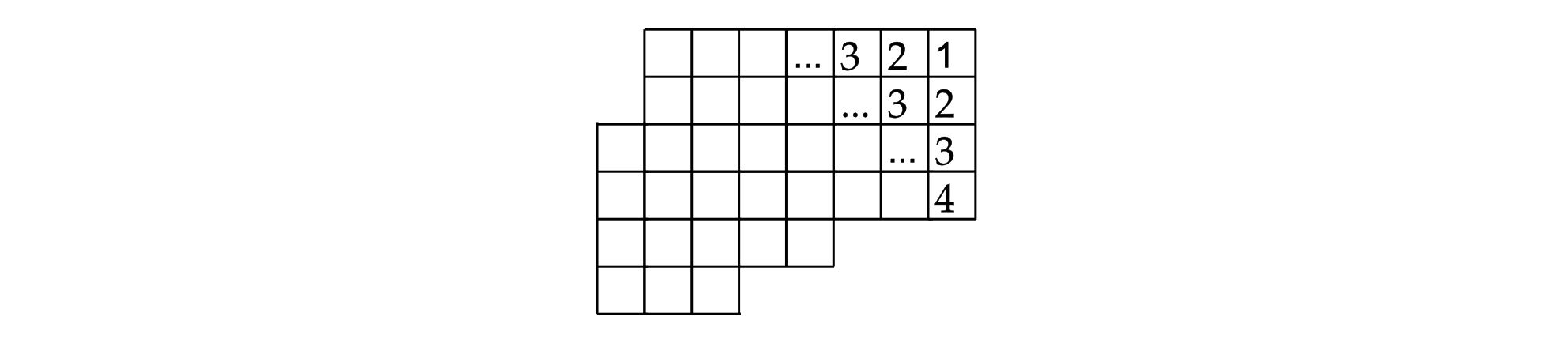}
\end{center}

In the following picture, notice that the filled box (green box) in the right skew shape is to the right and above the filled box (red box) in the left skew shape. Therefore, moving boxes to a higher position and to the right of the original position makes the $\lambda_1$-shifted diagonal index smaller. 
\begin{center}
    \includegraphics[width=\linewidth]{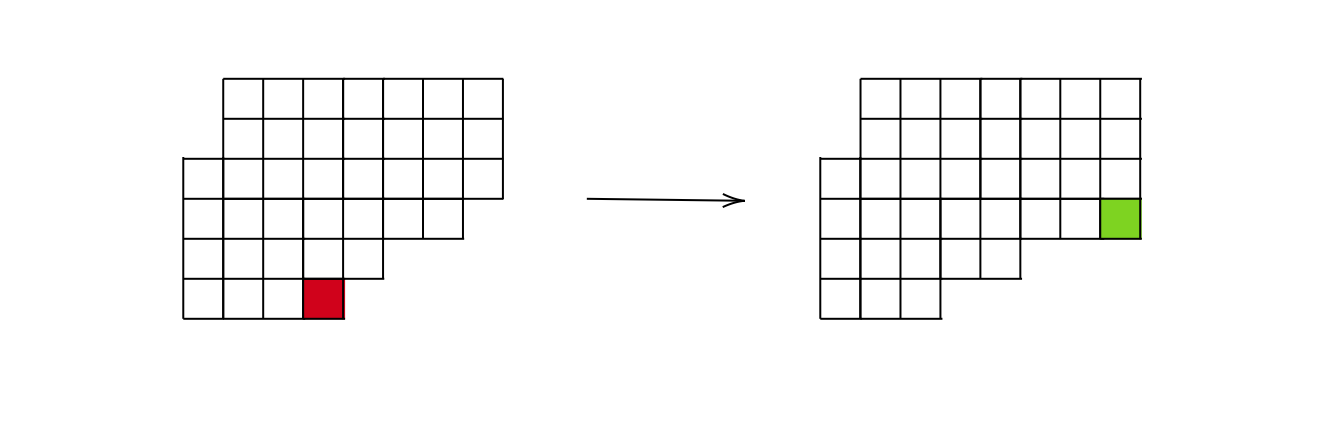}
\end{center}

Therefore, moving boxes to a higher row (without creating a new column) or to the right will allow us to bound $A_{\lambda^\downarrow}^{k}$ from below.

We consider the skew diagram obtained from $\lambda$ by removing the boxes that contain the labels $1, \ldots, n-k$. Since we are looking at the column insertion standard Young tableau $T_{\lambda^\downarrow}$, the resulting skew diagram is missing a few of the first columns of $\lambda$ and a few boxes from its first non--empty column. The rest of its columns of $\lambda$ is fully present.

We perform the above operation by moving the last corner of the diagram as far to the top and to the right as we can. We continue performing this operation until we can no longer do so, so that the outcome is still a skew diagram. At the end, we obtain a shape like the one below, which has four parameters $(p, q, r, s),$ where $r \leq p$ and $s < q$. Here $p$ denotes the number of columns in the full rows of the skew diagram, 
$q$ denotes the number of such full rows, $r$ denotes the number of extra boxes appended below the full rows, and $s$ denotes the number of rows that are missing boxes on the left.

\begin{center}
    \includegraphics[width=\linewidth]{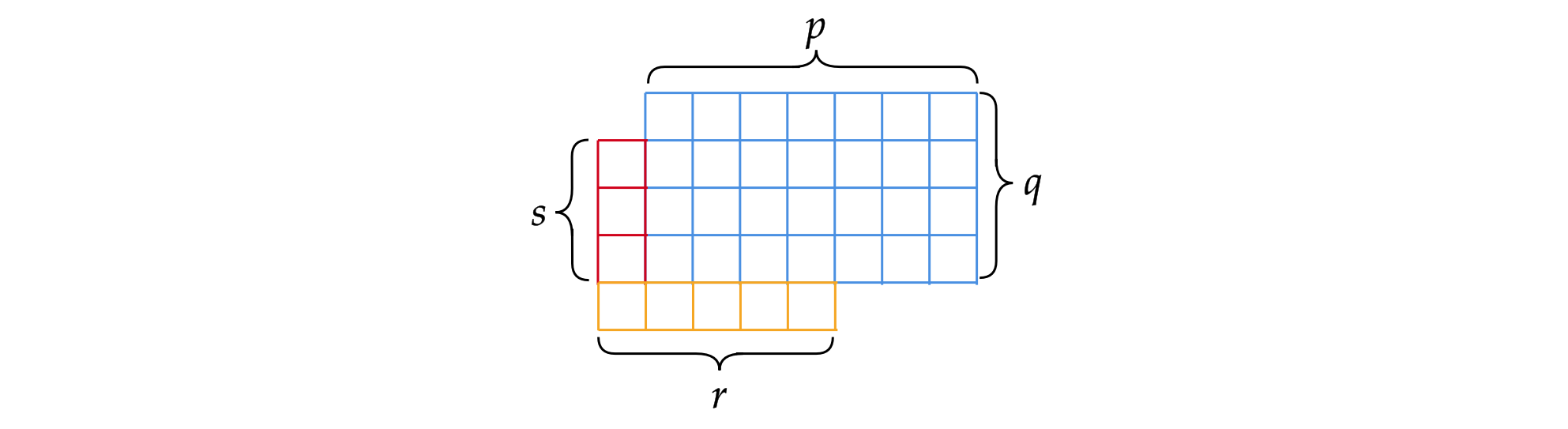}
\end{center}

 By symmetry, we can assume $q \leq p$. Also, we know $k = qp + r + s$, because the skew diagram has $k$ boxes in total. Since $\lambda$ is a Young diagram, we know that $\lambda_1 q \leq n$.

 We now prove a lower bound for $A_{\lambda^\downarrow}^{k}$.

\begin{lemma}\label{lowb1}

   Let $\lambda = (\lambda_1, \lambda_2, \dots, \lambda_m)$ be a partition of $n$ and $(q-1)^2 \geq r+s $. Then, $$A_{\lambda^\downarrow}^{k}\geq k+ \binom{k}{2}\frac{\lambda_1-1}{n-1},$$
   for every $1\leq k \leq n$.   
\end{lemma}
\begin{proof}
\begin{align*}
A_{\lambda^\downarrow}^k &= k + \sum_{j=1}^{p} \sum_{i=1}^{q} (i + j - 2) + \sum_{i=0}^{r-1} (p + q - i) + \sum_{i=1}^{s} (p + q - i) \\
&= k + \frac{1}{2} k (p + q - 2) + \frac{1}{2} r (p + q + 3 - r) + \frac{1}{2} s (p + q + 1 - s)
\end{align*}

Since $(q - 1)^2 \geq r + s$, we have $q - 2 \geq \frac{r + s - 1}{q}$.

Thus, we get
$$
A_{\lambda^\downarrow}^k \geq k + \frac{1}{2} k \left( p + \frac{r + s - 1}{q} \right) + \frac{1}{2} r(p + q + 3 - r) + \frac{1}{2} s(p + q + 1 - s).
$$

We also have
$$
p + \frac{r + s - 1}{q} = \frac{qp + r + s - 1}{q} = \frac{k - 1}{q}.
$$

Therefore, we obtain
$$
2 A_{\lambda^\downarrow}^k \geq 2k + k \frac{k - 1}{q} \geq 2k + k(k - 1) \frac{\lambda_1 - 1}{n - 1}.
$$

\end{proof}
The next lemma proves the same bound for the rest of the cases by reducing it to Lemma \ref{lowb1}.
\begin{lemma}\label{lowb2}

   Let $\lambda = (\lambda_1, \lambda_2, \dots, \lambda_m)$ be a partition of $n$ and $r+s> (q-1)^2$. Then, $$A_{\lambda^\downarrow}^{k}\geq k+ \binom{k}{2}\frac{\lambda_1-1}{n-1},$$
   for every $1\leq k \leq n$.   
\end{lemma}

\begin{proof}

By the parametrization above, $s \leq q - 1$. Therefore, $r > 0$ and $s \leq r$. So we can make the $(\lambda_1 - 1)$-shifted diagonal value smaller just by moving the top box (red box) in the first column of the following picture to the last box (green box) in the last column. As indicated in the figure, the $(\lambda_1 - 1)$-shifted diagonal value of the top box (red box) in the first column is $p + q - s$, while for the last box (green box) in the last row, it is $p + q - r$.

\begin{center}

    \includegraphics[width=\linewidth]{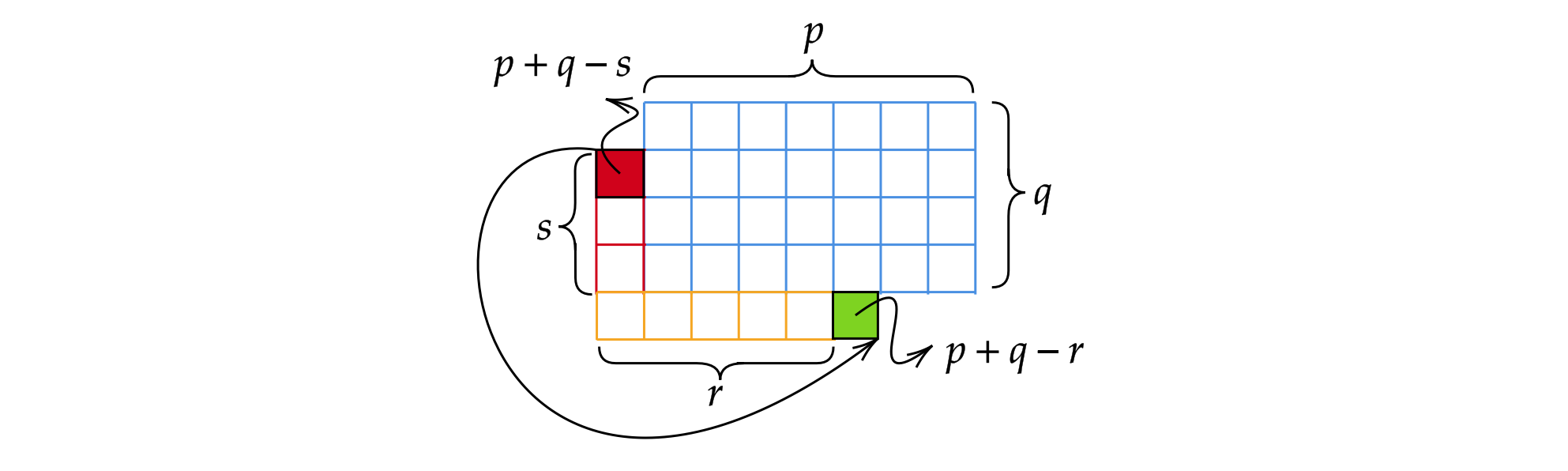}

 \end{center}
 
Now we continue to do this until we cannot continue any longer. There are two cases that can occur depending on the initial parameters $(p, q, r, s)$:

\begin{enumerate}
    \item If $r + s \geq p$, then we get a new shape $(p, q + 1, 0, r + s - p)$. The resulting shape (which has a smaller $(\lambda_1 - 1)$-shifted diagonal than the initial shape) is shown in the next figure. This new shape satisfies the conditions of Lemma \ref{lowb1}.

\begin{center}
    \includegraphics[width=\linewidth]{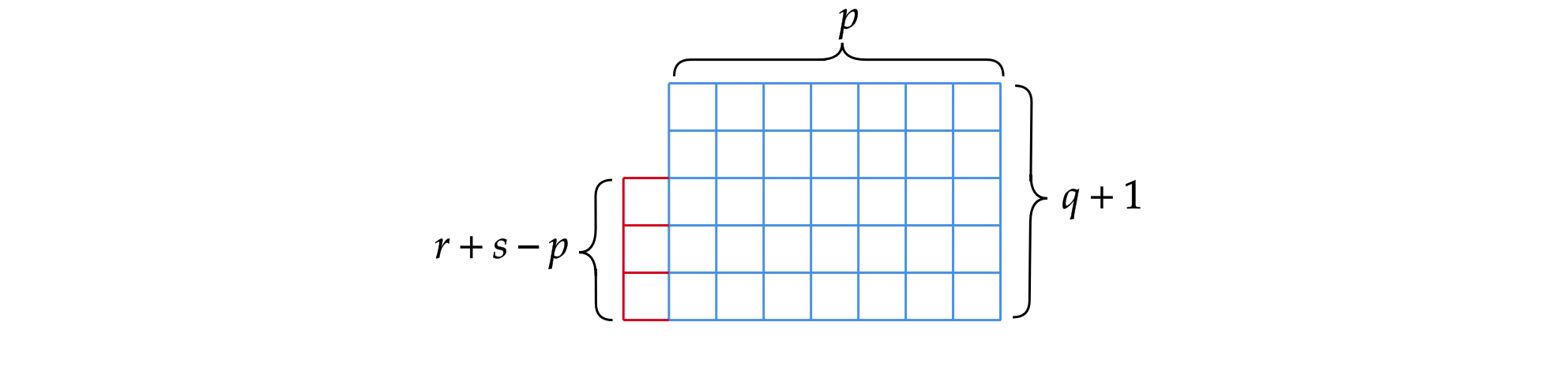}
 \end{center}

\item If $ r+s < p$, then the resulting shape after moving boxes from the first column to the last row is given by the following figure.

\begin{center}
    \includegraphics[width=\linewidth]{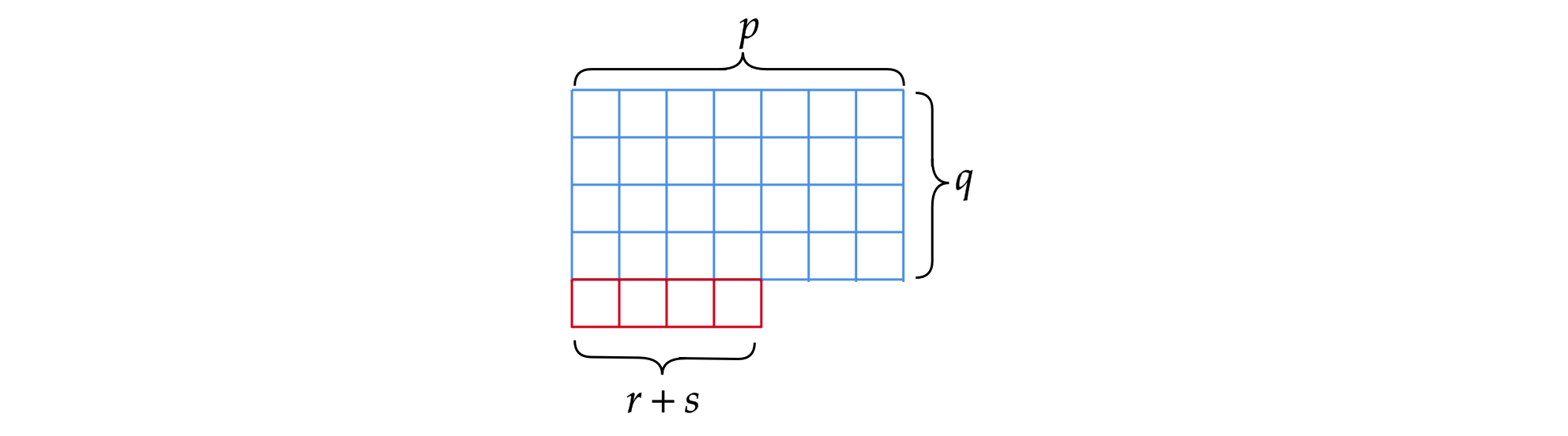}
 \end{center}

\end{enumerate}

But now we can still decrease $A_{\lambda^\downarrow}^k$ by moving boxes from the first column to the last row, because $p - 1 + q - r - s \leq p - 1$.

\begin{center}
    \includegraphics[width=\linewidth]{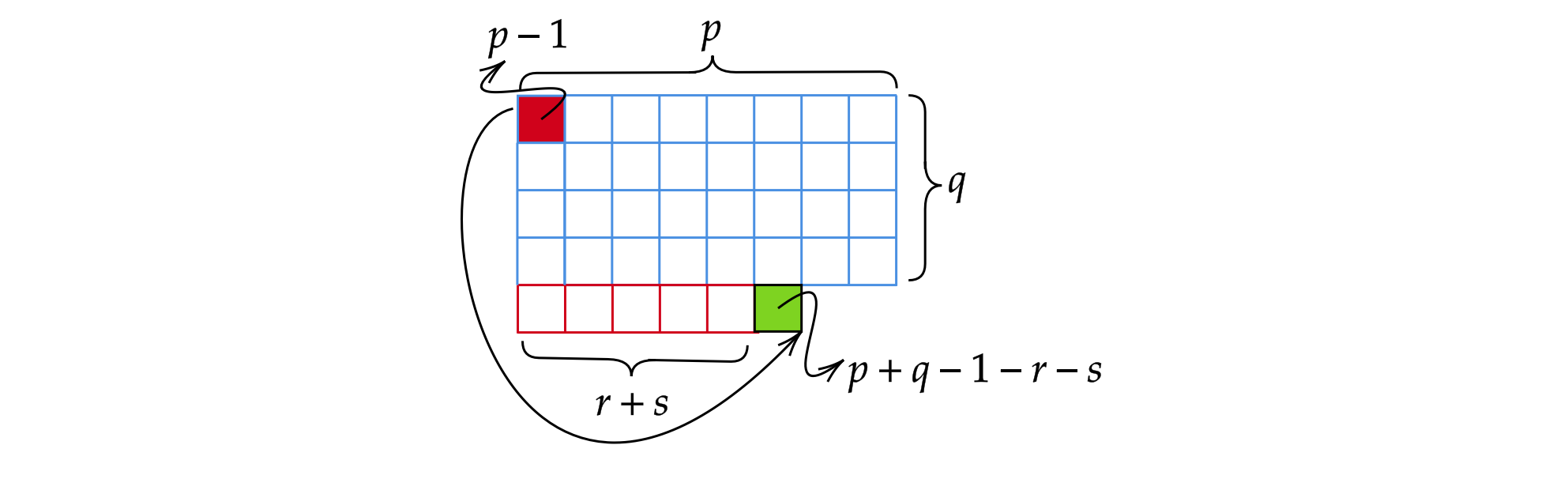}
 \end{center}

We continue this until we fill the $(q + 1)$-th row. 
In this case, we also end up with a new shape $(p, q + 1, 0, r + s - p)$.

\begin{center}
    \includegraphics[width=\linewidth]{diagram-5.png}
 \end{center}
Therefore, in both cases, we end up with the $(p, q + 1, 0, r + s - p)$ shape and we have $r + s - p \leq q^2$, which satisfies the conditions of Lemma \ref{lowb1}, whose $(\lambda_1 - 1)$-shifted diagonal value is smaller than the initial one. 

\end{proof}

We combine the statements of Lemmas \ref{lowb1} and \ref{lowb2} in the following lemma.
\begin{lemma}\label{MINAK}
   Let $\lambda = (\lambda_1, \lambda_2, \dots, \lambda_m)$ be a partition of $n$. Then, $$A_{\lambda^\downarrow}^{k}\geq k+ \binom{k}{2}\frac{\lambda_1-1}{n-1},$$
   for every $1\leq k \leq n$.   
\end{lemma}

We will now use Lemma \ref{MINAK} to provide bounds to the eigenvalues of $P_k$.
\begin{lemma}\label{UPbo}
    Let $\lambda = (\lambda_1, \lambda_2, \dots, \lambda_m) \vdash n$, $S \in SYT(\lambda)$, and $1\leq k \leq n$, then we have
    \begin{enumerate}
             
        \item $ \frac{2-m}{n} \leq \eig(S) \leq \frac{\lambda_1}{n} $, 
        \item $\mid \eig(S) \mid \leq 1-\frac{n-1}{n-\frac{k+1}{2}}\frac{n-\lambda_1}{n}\frac{\lambda_1+1}{n}$ if $ \lambda_1 > \frac{6n}{10}$ or $m > \frac{6n}{10}$.     \end{enumerate}
\end{lemma}

\begin{proof}
For part (i), using Lemma \ref{lem:eig-bounds}, we have
$$\eig\bigl(T_{\lambda^{\to}}\bigr) \leq \eig(S) \leq \eig\bigl(T_{\lambda^{\downarrow}}\bigr).$$
So for the upper bounds of i) and ii), we bound $\mathrm{eig}\bigl(T_{\lambda^{\downarrow}}\bigr)$.
Using Lemma \ref{MINAK}, Definition \ref{shifted}, and the fact that $\eig\bigl(T_{\lambda^{\downarrow}}\bigr) = \frac{1}{n} + \frac{2(n-1)}{nk(2n-(k+1))} D_{\lambda^\downarrow}^k$, we get
$$ \eig\bigl(T_{\lambda^{\downarrow}}\bigr) \leq \frac{1}{n} + \frac{(n-1)}{nk\left(n - \frac{k+1}{2}\right)}\left(k(\lambda_1 - 1) - \binom{k}{2} \frac{\lambda_1 - 1}{n-1}\right) = \frac{\lambda_1}{n}.$$
For the lower bound, let $S^T$ be the transpose of $S$. Since $D^k_S = -D^k_{S^T}$ and by the definition of $\eig(S)$, we have
\begin{equation} \label{tr}
    \eig(S) + \eig(S^T) = \frac{2}{n}.
\end{equation}
The upper bound of part (i) that we just proved says
\begin{equation}\label{m}
    \eig(S^T) \leq \frac{m}{n}.
\end{equation}
Equations \eqref{tr} and \eqref{m} give
$$\frac{2 - m}{n} \leq \eig(S).$$

Now, we prove part $(ii)$ for the case $\lambda_1 > \frac{6n}{10}$, since the case for $m > \frac{6n}{10}$ is similar. If $\lambda_1 > \frac{6n}{10}$, by part (i) we have $$\frac{-n+\lambda_1+1}{n}\leq \frac{2-m}{n} \leq \eig\bigl(T_{\lambda^{\to}}\bigr) \leq \eig(S) \leq  \eig\bigl(T_{\lambda^{\downarrow}}\bigr).$$ Since we have  $$\frac{n-\lambda_1-1}{n} \leq1-\frac{n-1}{n-\frac{k+1}{2}}\frac{n-\lambda_1}{n}\frac{\lambda_1+1}{n},$$ it suffices to only bound $\eig\bigl(T_{\lambda^{\downarrow}}\bigr)$ from above.

Case 1: If $k \leq \lambda_1 - \lambda_2$, then 
\begin{align*}
 \mathrm{eig}\bigl(T_{\lambda^{\downarrow}}\bigr)  &= \frac{1}{n} + \frac{(n-1)}{\left(n - \frac{k+1}{2}\right)} \frac{k(2\lambda_1 - (k+1))}{2nk} \\
&= 1 - \frac{(n-1)}{\left(n - \frac{k+1}{2}\right)} \frac{(n - \lambda_1)}{n} \\
& \leq 1 - \frac{(n-1)}{\left(n - \frac{k+1}{2}\right)} \frac{(n - \lambda_1)}{n} \frac{(\lambda_1 + 1)}{n}.
\end{align*}

Case 2: If $k > \lambda_1 - \lambda_2$, then by using the notation introduced in Lemma \ref{MINAK}, we are in the situation where $q = 1 \leq r$. Therefore, we have $r + s > (q - 1)^2$.

Using similar argument as in Lemma \ref{lowb1} and Lemma \ref{lowb2}, the maximum eigenvalue is attained when $\lambda = (\lambda_1, n - \lambda_1)$.

\begin{center}
    \includegraphics[width=\linewidth]{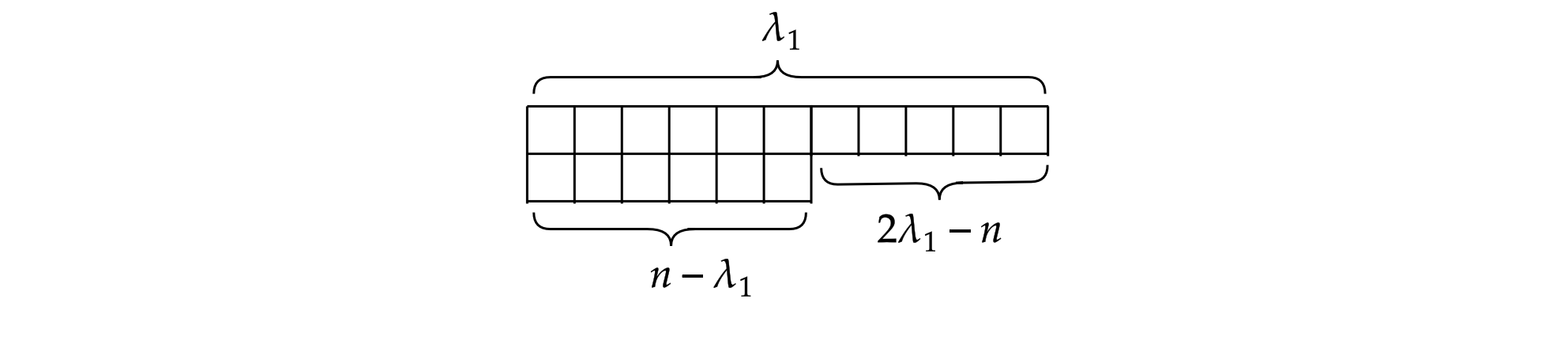}
 \end{center}

    Since $k> \lambda_1-\lambda_2$  we have $\frac{n+k}{2} > \lambda_1 > \frac{6n}{10}$. Therefore,

\begin{align*}
 \mathrm{eig}\bigl(T_{\lambda^{\downarrow}}\bigr) 
&\leq 1-\frac{(n-1)}{(n-\frac{k+1}{2})}\frac{(\lambda_1+1)(n-\lambda_1)-(\frac{n-k}{2})(\frac{n-k+2}{2})}{nk} \\
&\leq 1-\frac{(n-1)}{(n-\frac{k+1}{2})}\frac{(n-\lambda_1)}{n}\frac{(\lambda_1+1)}{n}. 
\end{align*}

\end{proof}

\section{The upper bound}\label{upper}
In this section, we present the analysis of \eqref{l2}. In particular, we provide upper bounds for 
 $$ \Sigma=\sum_{\lambda \neq (n)}d_{\lambda}\sum\limits_{\substack
 {\mu \vdash n-k \\ \mu \subseteq \lambda}}  d_{\mu}d_{\lambda / \mu } \bigg( \frac{1}{n} + \frac{(n-1)}{nk(n-\frac{k+1}{2})} \Big( \textup{Diag}(\lambda)-\textup{Diag}(\mu) \Big) \bigg)^{2t},$$  
when $t=t_{n,k}(c)$.
\begin{proof}[Proof of the upper bound]
We group the partitions of $n$ into the following zones, in order to treat eigenvalues with similar behavior.
\begin{enumerate}
\item[] $$Zone_1 :=\{\lambda : \lambda_1 \leq \frac{n}{3}, m \leq \frac{n}{3}\}$$
\item[] $$Zone_2:=\{\lambda : \frac{n}{3} < \lambda_1 \leq \frac{n}{2} \} \cup \{\lambda : \frac{n}{3} < m \leq \frac{n}{2} \}$$
\item[] $$Zone_3:=\{\lambda : \frac{n}{2} < \lambda_1 \leq \frac{6n}{10} \} \cup \{\lambda : \frac{n}{2} < m \leq \frac{6n}{10} \}$$
\item[] $$Zone_4:=\{\lambda : \frac{6n}{10} < m  \} $$
\item[] $$Zone_5:=\{\lambda : \frac{6n}{10} < \lambda_1 \} .$$
\end{enumerate}
This is summarized in the following picture.

\begin{center}
 
\tikzset{every picture/.style={line width=0.75pt}} 

\begin{tikzpicture}[x=0.75pt,y=0.75pt,yscale=-1,xscale=1]

\draw   (174,44.73) -- (398,239) -- (174,239) -- cycle ;
\draw   (174,135.73) -- (277.27,135.73) -- (277.27,239) -- (174,239) -- cycle ;
\draw   (174,192) -- (221,192) -- (221,239) -- (174,239) -- cycle ;
\draw    (174,89) -- (224,88.4) ;
\draw    (332,182) -- (332,239.4) ;

\draw (140,127.4) node [anchor=north west][inner sep=0.75pt]    {$n/2$};
\draw (140,181.4) node [anchor=north west][inner sep=0.75pt]    {$n/3$};
\draw (210,244.4) node [anchor=north west][inner sep=0.75pt]    {$n/3$};
\draw (265,245.4) node [anchor=north west][inner sep=0.75pt]    {$n/2$};
\draw (103,149.4) node [anchor=north west][inner sep=0.75pt]    {$m$};
\draw (244,285.4) node [anchor=north west][inner sep=0.75pt]    {$\lambda_{1}$};
\draw (139,81.4) node [anchor=north west][inner sep=0.75pt]    {$0.6n$};
\draw (313,246.4) node [anchor=north west][inner sep=0.75pt]    {$0.6n$};
\draw (180,66.4) node [anchor=north west][inner sep=0.75pt]    
{$4$};
\draw (206,108.4) node [anchor=north west][inner sep=0.75pt]    {$3$};
\draw (245,157.4) node [anchor=north west][inner sep=0.75pt]    {$2$};
\draw (193,206.4) node [anchor=north west][inner sep=0.75pt]    {$1$};
\draw (298,194.4) node [anchor=north west][inner sep=0.75pt]    {$3$};
\draw (341,213.4) node [anchor=north west][inner sep=0.75pt]    {$5$};

\end{tikzpicture}

\end{center}

For each zone, we can consider the terms
$$ \Sigma_i=\sum_{\lambda \in Zone_i}d_{\lambda}\sum\limits_{\substack
 {\mu \vdash n-k \\ \mu \subseteq \lambda}}  d_{\mu}d_{\lambda / \mu } \bigg( \frac{1}{n} + \frac{(n-1)}{nk(n-\frac{k+1}{2})} \Big( \textup{Diag}(\lambda)-\textup{Diag}(\mu) \Big) \bigg)^{2t}.$$  

Now, for each zone, we have a bound for the maximum eigenvalue by Lemma \ref{UPbo}. To bound the multiplicities of the eigenvalues, we will use the fact that $d_{\mu} d_{\lambda / \mu} \leq d_{\lambda}$ (which is a consequence of $\sum_{\mu \subseteq \lambda} d_{\lambda} d_{\mu} d_{\lambda / \mu} = d_{\lambda}^2$ and was also proven in \cite{luca}) and Lemmas \ref{L-DUP} and \ref{T-HR}. Zones 1, 2, and 3 are treated just as in inner, mid, and outer-Zone 1~\cite{Diaconis1981}. Namely,

\begin{enumerate}
    \item $\Sigma_1 \leq n! \left(\frac{1}{3}\right)^{2t} \leq b_1 e^{-2c}$,
    \item $\Sigma_2 \leq e^{\pi \sqrt{\frac{2n}{3}}} 4^n \left(\frac{2n}{3}\right)! \left(\frac{1}{2}\right)^{2t} \leq b_2 e^{-2c}$,
    \item $\Sigma_3 \leq e^{\pi \sqrt{\frac{2n}{3}}} 4^n \left(\frac{n}{2}\right)! \left(\frac{6}{10}\right)^{2t} \leq b_3 e^{-2c}$.
\end{enumerate}

For the above cases, the bound does not depend on $k$. Therefore, the fact that $t \geq \frac{1}{2} n (\log(n) + c)$ implies $\Sigma_i \leq Be^{-2c}$, for $i = 1, 2, 3$.

For Zones 4 and 5, we get more intricate bounds. In terms of the multiplicities, the bounds from the outer zone 2 and the outer zone 3~\cite{Diaconis1981} will apply. Equations (3.14) and (3.15)~\cite{Diaconis1981} prove that there exists a constant $b > 0$, universal in $n$, such that
\begin{equation}\label{eub}
e^{-2c}\sum_{j=0}^{0.4n} \frac{p(j)}{j!} e^{\frac{2j^2 \log(n)}{n}} \leq b_4 e^{-2c}.
\end{equation}
For large enough $n$ and  all $2\leq j \leq n$ we have $\frac{\log n}{0.4n} \leq \frac{\log j}{j} 
\Rightarrow 
\frac{2j^2 \log n}{n} \leq 0.8j \log j$.

Therefore, 

\begin{equation}
\sum_{2 \leq j \leq 0.4n} \frac{p(j)}{j!} 
e^{\frac{2j^2 \log n}{n}} 
\leq 
\sum_{j=2}^{\infty} \frac{p(j)}{j!} e^{0.8j \log j} 
< \infty,
\end{equation}

which implies 

\begin{align}
\log \left( \frac{p(j)}{j!} e^{0.8j \log j} \right) 
&= \log p(j) - \log(j!) + 0.8j \log j = -0.2j \log j + O(j).
\end{align}

In Zones 4 and 5, Lemma \ref{UPbo} gives that
$$
\left( \frac{1}{n} + \frac{(n-1)}{nk\left(n - \frac{k+1}{2}\right)} \left( \textup{Diag}(\lambda) - \textup{Diag}(\mu) \right) \right)^{2t} \leq e^{-2c} e^{\frac{2j^2 \log(n)}{n}},
$$
for $t = t_{n,k}(c)$. In total, we get
\begin{enumerate}
    \item[4.] $\Sigma_4 \leq e^{-2c} \sum_{j=0}^{0.4n} \frac{p(j)}{j!} e^{\frac{2j^2 \log(n)}{n}}$,
    \item[5.] $\Sigma_5 \leq e^{-2c} \sum_{j=1}^{0.4n} \frac{p(j)}{j!} e^{\frac{2j^2 \log(n)}{n}}$.
\end{enumerate}
Cases 1-3 and \eqref{eub} give
$$\Sigma \leq e^{-2c}(b_1 + b_2 + b_3 + 2b_4) = a^2 e^{-2c},$$
and this finishes the proof of the upper bound.

\end{proof}

\section{The lower bound}\label{lower}
In this section, we prove the lower bound of Theorem \ref{T-Cutoff} for the set $E=\{n-k+1,\ldots, n\}$. We use the abbreviated notation $P_{id}^t(y)$ instead of $P_k^t(id,y)$ throughout this section.

We use the second moment method, introduced by Diaconis (see, for example, Exercise 13 on page 44~\cite{Diaconis1981}). Let $\chi_{(n-1,1)}$ be the character corresponding to the partition $\lambda = (n-1,1)$. To study how $\chi_{(n-1,1)}$ concentrates with respect to $P_{id}^t$, we will compute $ \text{Var}_{P_{id}^t}(\chi_{(n-1,1)}) $.

Just as in Exercise 13 on page 44~\cite{Diaconis1981}, we have
$$
\chi_{(n-1,1)}^2 = \chi_{(n)} + \chi_{(n-1,1)} + \chi_{(n-2,2)} + \chi_{(n-2,1,1)},
$$
and therefore
$$
\text{Var}_{P_{\text{id}}^t}(\chi_{(n-1,1)}) = \mathbb{E}_{P_{\text{id}}^t}(\chi_{(n)}) + \mathbb{E}_{P_{\text{id}}^t}(\chi_{(n-1,1)}) + \mathbb{E}_{P_{\text{id}}^t}(\chi_{(n-2,2)}) + \mathbb{E}_{P_{\text{id}}^t}(\chi_{(n-2,1,1)}) - \mathbb{E}_{P_{\text{id}}^t}(\chi_{(n-1,1)})^2.
$$

To compute the expectations
$$
\mathbb{E}_{P_{\text{id}}^t}(\chi_{(n)}), \quad \mathbb{E}_{P_{\text{id}}^t}(\chi_{(n-1,1)}), \quad \mathbb{E}_{P_{\text{id}}^t}(\chi_{(n-2,2)}), \quad \mathbb{E}_{P_{\text{id}}^t}(\chi_{(n-2,1,1)}),
$$
we simply need the corresponding eigenvalues. 


\begin{lemma}\label{L-low} 
For the $k$--star transpositions, we have
\begin{align*}
\mathbb{E}_{P_{id}^t}(\chi_{(n)}) &= 1,\\
\mathbb{E}_{P_{id}^t}(\chi_{(n-1,1)}) &= 
    k \left( 1 - \frac{n-1}{n} \cdot \frac{n}{k \left( n - \frac{k+1}{2} \right)} \right)^t 
    + (n-1-k) \left( 1 - \frac{n-1}{n} \cdot \frac{k}{k \left( n - \frac{k+1}{2} \right)} \right)^t, \\
\mathbb{E}_{P_{id}^t}(\chi_{(n-2,2)}) &= 
    \binom{k}{2} \left( 1 - \frac{n-1}{n} \cdot \frac{2(n-1)}{k \left( n - \frac{k+1}{2} \right)} \right)^t 
    + k (n-1-k) \left( 1 - \frac{n-1}{n} \cdot \frac{k+n-2}{k \left( n - \frac{k+1}{2} \right)} \right)^t \\
    &\quad + \frac{(n-k)(n-3-k)}{2} \left( 1 - \frac{n-1}{n} \cdot \frac{2k}{k \left( n - \frac{k+1}{2} \right)} \right)^t, \mbox{ and}\\
\mathbb{E}_{P_{id}^t}(\chi_{(n-2,1,1)}) &= 
    \binom{k}{2} \left( 1 - \frac{n-1}{n} \cdot \frac{2n}{k \left( n - \frac{k+1}{2} \right)} \right)^t 
    + k (n-1-k) \left( 1 - \frac{n-1}{n} \cdot \frac{k+n}{k \left( n - \frac{k+1}{2} \right)} \right)^t \\
    &\quad + \binom{n-1-k}{2} \left( 1 - \frac{n-1}{n} \cdot \frac{2k}{k \left( n - \frac{k+1}{2} \right)} \right)^t .
\end{align*}
\end{lemma}
The following figure gives the relevant eigenvalues and their multiplicities.
\begin{center}
 \includegraphics[width=420pt]{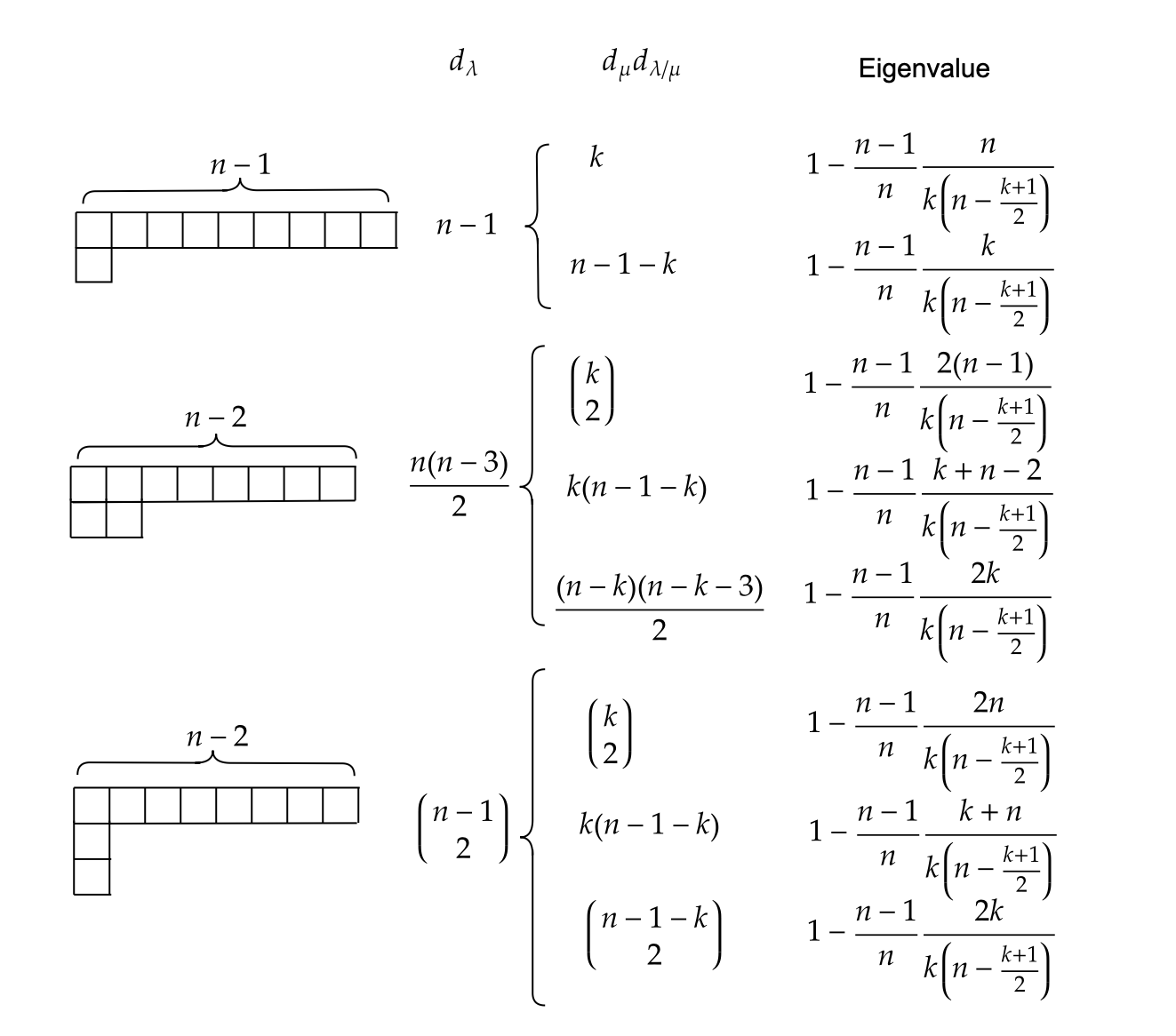}
 \end{center}

\begin{proof}
  These computations follow from the fact that $\mathbb{E}_{P_{id}^t}(\chi_{\lambda})=\operatorname{Tr}(\hat{P}(\rho_{\lambda}))$, where $\lambda$ is a partition of $n$, $\rho_{\lambda}$ is the corresponding irreducible representation, and $$\hat{P}(\rho_{\lambda}):= \sum_{x \in S_n} P(id, x)\rho_{\lambda}(x)$$ is the Fourier transform of $P$ at $\rho$. A standard fact is that the eigenvalues of $\hat{P}(\rho_{\lambda})$ are given exactly by the eigenvalues of $P$ with respect to $\lambda$ (see Theorem 6, Chapter 3E from \cite{Diaconis1988}).
\end{proof}

In the case where $\lim_{n \rightarrow \infty} \frac{k}{n}=0$ or $1$ and $t=\frac{2n-(k+1)}{2(n-1)}n(\log n+c)$, Lemma \ref{L-low} implies 
$$\lim_{n\to \infty}\text{Var}_{P_{id}^{t}(\chi_{(n-1,1)})}=\lim_{n\to \infty}1+\mathbb{E}_{P_{id}^{t}}(\chi_{(n-1,1)})=1+e^{-c}.$$

\begin{proof}[Proof of the lower bound]

Let us consider the set 
$F_{\ell} = \{\, \sigma \in S_n \mid |\chi_{(n-1,1)}(\sigma)| \leq \ell \,\}$
for any $ \ell > 0 $. It is known that $ \chi_{(n-1,1)}(\sigma) = |\textup{fix}(\sigma)| - 1 $, where $ \textup{fix}(\sigma) $ denotes the number of fixed points of the permutation $ \sigma $. Thus, the following inequality holds:
$$
\Vert P_{\text{id}}^{t} - U \Vert \geq |P_{\text{id}}^{t}(F_{\ell}) - U(F_{\ell})|.
$$

Next, consider the estimate for $ P_{\text{id}}^{t}(F_{\ell}) $:

$$
P_{\text{id}}^{t}(F_{\ell}) \leq P_{\text{id}}^{t}\left( |\chi_{(n-1,1)} - \mathbb{E}_{P_{\text{id}}^{t}}(\chi_{(n-1,1)})| \geq \mathbb{E}_{P_{\text{id}}^{t}}(\chi_{(n-1,1)}) - {\ell} \right) \leq \frac{\text{Var}_{P_{\text{id}}^{t}}(\chi_{(n-1,1)})}{(\mathbb{E}_{P_{\text{id}}^{t}}(\chi_{(n-1,1)}) - {\ell})^2}.
$$

Let $d(n,i)$ be the number of permutations in $S_n$ that have exactly $i$ fixed points. For the uniform measure, we can express $ U(F_{\ell}) $ as:

$$
U(F_{\ell}) = \frac{1}{n!} \sum_{i=0}^{\lfloor {\ell} \rfloor + 1} d(n,i) = \frac{1}{n!} \sum_{i=0}^{\lfloor {\ell} \rfloor + 1} \binom{n}{i}!(n-i)
= \frac{1}{n!} \sum_{i=0}^{\lfloor {\ell} \rfloor + 1} \binom{n}{i} \left\lfloor \frac{(n-i)!}{e} + \frac{1}{2} \right\rfloor \geq 1 - \frac{1}{e{\ell}},
$$

for sufficiently large $ n $.

Setting $ {\ell} := \frac{e^{-c}}{2} $, we obtain:

$$
|P_{\text{id}}^{t}(F_{\ell}) - U(F_{\ell})| \geq U(F_{\ell}) - P^{t}(F_{\ell}) \geq 1 - \frac{1}{e{\ell}} - \frac{1 + 2{\ell}}{{\ell}^2}.
$$

If $ c $ is chosen so that $\frac{1}{e{\ell}}+ \frac{1 + 2{\ell}}{{\ell}^2}\leq \varepsilon $, then it follows that:

$$
|P^{t}(F_{\ell}) - U(F_{\ell})| \geq U(F_{\ell}) - P^{t}(F_{\ell}) \geq 1 - \varepsilon.
$$

\end{proof}

\section{The limit profile}\label{lp}

In this section, we present the proof of Theorem \ref{T-LP}. We use Lemma \ref{comp} to compare the limit profile of the $k$--star transpositions with the limit profile of random transpositions. We are able to make this comparison because the transition matrix of random transpositions and the transition matrix of the $k$--star transpositions share the same eigenbasis, as required by Lemma \ref{comp}. This generalizes the known results for random transpositions \cite{Teyssier2019}, which correspond to the identity function $f(n) = n$, and for the star transpositions \cite{nestoridi2024comparing}, which correspond to the constant function $f(n) = 1$.


\begin{lemma}\label{L-SK}
Let $\lambda$ be a partition of $n$ and $j=n-\lambda_1  $. Also, let $\mu$ be a partition of $n-k$ and set $\ell=k-\lambda_1+\mu_1 >0 $. Then we have 

\begin{align*}
       d_{\mu}d_{\lambda / \mu}  \leq  \Big(\frac{4^jk}{n}\Big)^{\ell}   d_{\lambda}.\end{align*}

\end{lemma}

\begin{proof}
This proof follows the argument in Lemma 5.3 of~\cite{nestoridi2024comparing}, the only difference being the removal of $\ell$ boxes instead of one. 
Specifically, we remove $ \lambda_1 - \mu_1 $ boxes from the first row, which implies that $ \ell $ boxes are removed from rows $ i > 1 $. Therefore,
$$
d_{\mu} \leq \prod_{i=0}^{\ell-1} \left(\frac{4^{j-i}}{n-i}\right) d_{\lambda}.
$$
It is also clear that
$$
d_{\lambda / \mu} \leq \ell! \binom{k}{\ell}.
$$
Hence,
$$
d_{\mu} d_{\lambda / \mu} \leq \prod_{i=0}^{\ell-1} \left(\frac{4^{j-i}}{n-i}\right) \ell! \binom{k}{\ell} d_{\lambda} \leq \left(\prod_{i=0}^{\ell-1} \frac{k - i}{n - i} \right) 4^{j\ell} d_{\lambda} \leq \left( \frac{k 4^{j}}{n} \right)^{\ell} d_{\lambda}.
$$

\end{proof}

\begin{proof}[Proof of Theorem~\ref{T-LP}]
We prove that the right--hand side of the inequalities in Lemma \ref{comp} is equal to zero. Since the limit profile for random transpositions exists, it follows that the limit profile for the $k$--star shuffles exists and is exactly equal to the limit profile of random transpositions. 
We denote $s_{(\lambda, \mu)} :=\eig(\lambda,\mu)$. 

Case 1: $\lim_{n \to \infty}\frac{k}{n}=0$

We adopt the notation of Theorem 1.3 in \cite{nestoridi2024comparing}. Here $\lambda^{'}$ denotes the conjugate partition of $\lambda$, where $\lambda^{'}_1$ equals the number of rows of $\lambda$.
We claim that there exists an $M=M(c,\varepsilon)$ such that 

\begin{enumerate}

    \item $\sum_{\lambda_1, \lambda_1^{'}\leq n-M} d^{2}_{\lambda}\mid s_{\lambda}\mid^{2t_{n,n}} \leq \varepsilon$
    \item $\sum_{\lambda_1, \lambda_1^{'}\leq n-M}d_{\lambda}\sum_{(\lambda,\mu)}d_{\mu}d_{\lambda/\mu}|s_{(\lambda, \mu)}|^{2t_{n,k}} \leq \varepsilon$
    \item $\sum_{\lambda_1 > n-M}d_{\lambda}\sum_{(\lambda,\mu)}d_{\mu}d_{\lambda/\mu}|s_{\lambda}^{t_{n,n}}-s_{(\lambda, \mu)}^{t_{n,k}}|^{2}\leq \varepsilon$
    \item $\sum_{\lambda_1^{'} > n-M}d_{\lambda}\sum_{(\lambda,\mu)}d_{\mu}d_{\lambda/\mu}|s_{\lambda}^{t_{n,n}}-s_{(\lambda, \mu)}^{t_{n,k}}|^{2} \leq \varepsilon$ 
    
\end{enumerate}

for sufficiently large n. 

The first part follows from Lemma 4.1 of Teyssier \cite{Teyssier2019}. Therefore, there exists an $ M_1 = M_1(c, \varepsilon) $ such that the first part holds.

For the second part, by the same argument as in \cite{nestoridi2024comparing}, there exists an $ M_2 = M_2(c, \varepsilon) $ such that
$$
\sum_{j \geq M_2} \frac{e^{-2cj}}{j!} < \varepsilon.
$$
By Lemma \ref{UPbo}, we have
$$
\frac{2 - m}{n} \leq s_{(\lambda, \mu)} \leq \frac{\lambda_1}{n},
$$
and since $ m \leq j + 1 $,
$$
|s_{(\lambda, \mu)}| \leq 1 - \frac{j}{n}.
$$
Because $ \lim_{n \to \infty} \frac{k}{n} = 0 $, for sufficiently large $ n $,
$$
\sum_{\lambda_1, \lambda_1^{\prime} \leq n - M_2} d_{\lambda} \sum_{(\lambda, \mu)} d_{\mu} d_{\lambda / \mu} |s_{(\lambda, \mu)}|^{2t_{n,k}} \leq \sum_{j \geq M_2} d_{\lambda}^2 \left(1 - \frac{j}{n}\right)^{2t_{n,k}} \leq 
\sum_{j \geq M_2} \frac{e^{-2cj}}{j!} \leq \varepsilon.
$$
This concludes part 2.

Now let $ M = \max(M_1, M_2) $. Using parts 1 and 2, along with the inequality $ (a^t - b^t)^2 \leq 2(a^{2t} + b^{2t}) $, we obtain

$$\sum_{\lambda_1, \lambda_1^{'}\leq n-M}d_{\lambda}\sum_{(\lambda,\mu)}d_{\mu}d_{\lambda / \mu}|s_{\lambda}^{t_{n,n}}-s_{(\lambda, \mu)}^{t_{n,k}}|^{2}\leq \varepsilon.$$

We now prove part 3 by dividing the sum into two cases. Part 4 follows by an analogous argument.


\begin{enumerate}
\item[(a)] 

$$\sum_{ \lambda_1 > n-M}d_{\lambda}\sum_{\substack{(\lambda,\mu)\\
                  \lambda_1-\mu_1 <k}}d_{\mu}d_{\lambda / \mu}|s_{\lambda}^{t_{n,n}}-s_{(\lambda, \mu)}^{t_{n,k}}|^{2} $$

\item[(b)] 

$$\sum_{\lambda_1 > n-M}d_{\lambda}\sum_{\substack{(\lambda,\mu)\\
                  \lambda_1-\mu_1 =k}}d_{\mu}d_{\lambda / \mu}|s_{\lambda}^{t_{n,n}}-s_{(\lambda, \mu)}^{t_{n,k}}|^{2}. $$

\end{enumerate}

For (a), applying Lemma \ref{UPbo} for sufficiently large $ n $,
$$
|s_{\lambda}|^{t_{n,n}} \leq 3\frac{e^{-cj}}{n^j} \quad \text{and} \quad |s_{(\lambda,\mu)}|^{t_{n,k}} \leq \left(1-\frac{j}{n}\right)^{t_{n,k}} \leq \frac{e^{-cj}}{n^j}.
$$
Using Lemma \ref{L-SK}, along with the fact that
$$
\lim_{n \to \infty}\frac{k}{n}=0 \quad \Rightarrow \quad \lim_{n \to \infty}\left(\frac{4^M k}{n}\right)^{\ell}=0 \quad \text{for any } \ell \geq 1,
$$
it follows that
$$
\sum_{ \lambda_1 > n-M}\sum_{\substack{(\lambda,\mu)\\
\lambda_1 - \mu_1 < k}} d_{\lambda} d_{\mu} d_{\lambda / \mu} |s_{\lambda}^{t_{n,n}} - s_{(\lambda, \mu)}^{t_{n,k}}|^{2} \leq \sum_{\ell=1}^{\min(k,M)} M! \binom{M}{\ell} \left(\frac{4^M k}{n}\right)^{\ell} 4^2 \sum_{j<M} \frac{e^{-2cj}}{j!} \leq \varepsilon.
$$

For part (b), note that
$$
s_{(\lambda,\mu)}^{t_{n,k}} = \left(1 - \frac{n-1}{n - \frac{k+1}{2}} \frac{j}{n}\right)^{t_{n,k}} = \frac{e^{-cj}}{n^j} \left(1 + O\left(\frac{j^2}{n}\right)\right),
$$
and
$$
s_{\lambda}^{t_{n,n}} = \frac{e^{-cj}}{n^j} \left(1 + O\left(\frac{\log(n)}{n}\right)\right).
$$
Then, by Lemma \ref{L-DUP},
$$
\sum_{\lambda_1 > n-M} d_{\lambda} \sum_{\substack{(\lambda,\mu)\\
\lambda_1 - \mu_1 = k}} d_{\mu} d_{\lambda / \mu} |s_{\lambda}^{t_{n,n}} - s_{(\lambda, \mu)}^{t_{n,k}}|^{2} \leq \varepsilon,
$$
which completes the proof in the case $ \lim_{n \to \infty} \frac{k}{n} = 0 $.

Case 2: $\lim_{n \to \infty}\frac{k}{n}=1$ 

In this case, we have 
$$s_{(\lambda,\mu)}=1-\frac{n-1}{n-\frac{k+1}{2}}\frac{j}{n}+O\left(\frac{1}{n^2}\right),$$
which implies   \/
$$s_{(\lambda,\mu)}^{t_{n,k}}=\frac{e^{-cj}}{n^j}\left(1+O\Big(\frac{\log(n)}{n}\Big)\right).$$
Consequently, the difference satisfies 
$$\mid s_{\lambda}^{t_{n,n}}-s_{(\lambda, \mu)}^{t_{n,k}}\mid=\frac{e^{-cj}}{n^j}O\left(\frac{\log(n)}{n}\right).$$
Therefore, we obtain $$\sum_{ \lambda_1 > n-M}d_{\lambda}\sum_{\substack{(\lambda,\mu)}}d_{\mu}d_{\lambda / \mu}|s_{\lambda}^{t_{n,n}}-s_{(\lambda, \mu)}^{t_{n,k}}|^{2}=O\left(\frac{\log^2(n)}{n^2}\right),$$
which is sufficient to complete the proof of Case 2.
\end{proof}







\printbibliography

\end{document}